\documentclass[a4j]{article}
\usepackage{lipsum}
\usepackage[numbers]{natbib}
\usepackage{amsmath,amsthm,amssymb}
\usepackage{graphicx}
\usepackage{hhline,url}
\usepackage{mathrsfs}
\usepackage{xcolor}

\usepackage{algorithm,algpseudocode}

\algnewcommand{\Initialize}[1]{%
  \State \textbf{Initialize:}
  \State \hspace*{\algorithmicindent}\parbox[t]{0.8\linewidth}{\raggedright #1}
}

\newtheorem{thm}{Theorem}[]
\newtheorem*{thm*}{Theorem}
\newtheorem{m-thm}[thm]{Meta-Theorem}
\newtheorem*{m-thm*}{Meta-Theorem}
\newtheorem{lem}{Lemma}[]
\newtheorem{remark}{Remark}[]
\newenvironment{rem}{\begin{remark}\rm}{\end{remark}}
\newtheorem{prop}{Proposition}[]
\newtheorem*{prop*}{Proposition}
\newtheorem{Definition}{Definition}
\newenvironment{dfn}{\begin{Definition}\rm}{\end{Definition}}
\newtheorem{Corollary}[]{Corollary}
\newenvironment{cor}{\begin{Corollary}}{\end{Corollary}}
\newtheorem{Example}[]{Example}
\newenvironment{eg}{\begin{Example}\rm}{\end{Example}}
\newtheorem{algor}[thm]{Method}

\newtheorem{Condition}[thm]{Condition}

\newtheorem{assp}{Assumption}[]

\newtheorem{asspp}{Assumption}[]

\newcommand{\Hil}{\mathcal{H}}

\newcommand{\R}{\mathbb{R}}
\newcommand{\dd}{\,\mathrm{d}}
\newcommand{\ve}{\varepsilon}

\newcommand{\F}{\mathcal{F}}

\renewcommand{\phi}{\varphi}

\newcommand{\X}{\mathcal{X}}
\newcommand{\bm}[1]{{\mbox{\boldmath $#1$}}}

\DeclareMathOperator{\cv}{conv}

\DeclareMathOperator{\supp}{supp}

\newcommand{\lmid}{\,\middle|\,}

\newcommand{\E}[1]{\mathbb{E}\!\left[#1\right]}
\renewcommand{\P}[1]{\mathbb{P}\!\left(#1\right)}
\newcommand{\ord}[1]{\mathcal{O}\!\left(#1\right)}
\newcommand{\K}{\mathcal{K}}

\newcommand{\vertiii}[1]{{\left\vert\kern-0.25ex\left\vert\kern-0.25ex\left\vert #1 
    \right\vert\kern-0.25ex\right\vert\kern-0.25ex\right\vert}}
\usepackage{bbm}
\usepackage{xcolor}
\usepackage{todonotes}
\usepackage{mathtools}

\usepackage{caption}
\usepackage{subcaption}
\usepackage[inline]{enumitem}
\renewcommand{\tilde}{\widetilde}
\DeclareMathOperator{\wce}{wce}
\usepackage{bbm}
\newcommand{\p}{\mathbb{P}}

\usepackage{comment}
\newcommand{\tr}{\mathop\mathrm{tr}}

\usepackage[utf8]{inputenc} 
\usepackage[T1]{fontenc}    
\usepackage{url}            
\usepackage{booktabs}       
\usepackage{amsfonts}       
\usepackage{nicefrac}       
\usepackage{microtype}      
\usepackage{xcolor}         

\title{Hypercontractivity Meets Random Convex Hulls:\\
Analysis of Randomized Multivariate Cubatures}

\author{%
    Satoshi Hayakawa, \ Harald Oberhauser, \ 
    Terry Lyons\\
    Mathematical Institute,
    University of Oxford\\
    \texttt{\{hayakawa,oberhauser,tlyons\}@maths.ox.ac.uk}
}
\date{}

\begin{document}

\maketitle

\begin{abstract}
Given a probability measure $\mu$ on a set $\X$ and a vector-valued function $\bm \phi$, a common problem is to construct a discrete probability measure on $\X$ such that the push-forward of these two probability measures under $\bm\phi$ is the same.
This construction is at the heart of numerical integration methods that run under various names such as quadrature, cubature, or recombination. 
A natural approach is to sample points from $\mu$ until their convex hull of their image under $\bm \phi$ includes the mean of $\bm \phi$.
Here we analyze the computational complexity of this approach when $\bm\phi$ exhibits a graded structure by using so-called hypercontractivity.
The resulting theorem not only covers the classical cubature case of multivariate polynomials, but also integration on pathspace, as well as kernel quadrature for product measures.
\end{abstract}

\section{Introduction}


Let $X$ be a random variable that takes values in a set $\X$, and $\F \subset \R^\X$ a linear, finite dimensional space of integrable functions from $\X$ to $\R$. 
A cubature formula for $(X,\F)$ is a finite set of points $(x_i)\subset \X$ and weights $(w_i)\subset  \R$ such that
\begin{equation}
    \E{f(X)}
    = \sum_{i=1}^n w_if(x_i)\text{ for all }f\in \F.
    \label{eq:cubature}
\end{equation}
If the function class $\F$ is infinite-dimensional one can not hope for equality in \eqref{eq:cubature} and instead aims to find an approximation that holds uniformly over $\F$.
We also denote $\mu=\operatorname{Law}(X)$ and refer to $\hat \mu = \sum_{i=1}^n w_i \delta_{x_i}$ as the cubature measure for $(X,\F)$.
The existence of such a cubature formula that further satisfies $n \le 1 + \dim P$, $w_i\ge0$ and $\sum_{i=1}^n w_i=1$ is guaranteed by what is often referred to as Tchakaloff's theorem although a more accurate nomenclature would involve Wald, Richter, Rogosinski, and Rosenbloom \cite{wald1939limits,tch57,richter1957parameterfreie,rogosinski1958moments,rosenbloom1952quelques}; see \cite{Dio2018TheMT} for a historical perspective. 
Arguably the most famous applications concerns the case when $\X$ is a subset of $\R^d$ and $\F$ is the linear space of polynomials up to a certain degree, that is $\F$ is spanned by monomials up to a certain degree.
However, more recent applications include the case when $\X$ is a space of paths and $\F$ is spanned by iterated Ito--Stratonovich integrals \cite{lyo04}, or kernel quadrature \citep{kar19,hayakawa21b}
where $\X$ is a set that carries a positie definite kernel and $\F$ is a subset of the associated reproducing kernel Hilber space that is spanned by eigenfunctions of the integral operator induced by a kernel.

\paragraph{Convex Hulls.}
If $\F$ is spanned by $m$ functions $\phi_1,\ldots,\phi_m:\X \to \R$, then we can denote $\bm\phi = (\phi_1, \ldots, \phi_m) :\X\to\R^m$ and see that \eqref{eq:cubature} is equivalent to
\[\E{\bm\phi(X)} = \sum_{i=1}^n w_i\bm\phi(x_i).\]
If we restrict attention to  non-negative weights that sum up to one (equivalently, $\hat \mu$ is a probability measure) this is equivalent to that statement that
\begin{equation}
    \E{\bm\phi(X)}
    \in \cv\{\bm\phi(x_1), \ldots, \bm\phi(x_n)\},
    \label{eq:cv}
\end{equation}
where we denote for an $A\subset\R^m$ its convex hull as \[
\cv A =\left\{c_1a_1+\cdots+c_ka_k \lmid
k\ge1,\, a_i\in A,\, c_i\ge0,\, \sum_{i=1}^kc_i=1\right\}.
\]
\paragraph{Random Convex Hulls.}
A natural and general approach to find points $(x_i) \subset \X$ for which \eqref{eq:cv} holds was recently proposed in \citep{hayakawa-MCCC}: draw $N\gg n$ independent random samples $(X_j)_{j=1}^N$ from $\mu$ and subsequently try to select a subset of $n$ points $(x_i)$. 
The success of this approach amounts the event that 
\begin{equation}
    \E{\bm\phi(X)}
    \in \cv\{\bm\phi(X_1), \ldots, \bm\phi(X_N)\},
    \label{eq:rcv}
\end{equation}
since then simple linear programming (LP) allows select the subset of $x_i$'s resp.~compute the remaining weights that determine a cubature formula.
The following guarantees that for large enough $N$ this event occurs with high probability
\begin{prop}[\citep{hayakawa-MCCC}]\label{prop:cvx event}
    If $X_1, X_2, \ldots$ are independent copies of $X$,
    then the probability of the event \eqref{eq:rcv}
    tends to $1$ as $N\to\infty$.
\end{prop}
Empirically, this approach turns out to work well already for ``reasonable'' magnitudes of $N$ \citep{hayakawa-MCCC,hayakawa21b}. 
The aim of this article is to fill this gap and provide theoretical guarantees for the number of samples $N$ for which this approach leads with high probability to a successful cubature construction; that is to provide a quantitative version of Proposition \ref{prop:cvx event} that applies to common cases.

\paragraph{Hypercontractivity.}
Our main tool is hypercontractivity.
This allows to prove the existence of a constant $C_m^\prime$ satisfying (mainly for $p=4$)
\[
    \E{\lvert f(X)\rvert^p} \le C_m^\prime \E{\lvert f(X)\rvert^2}^{p/2}
\]
uniformly for a large class of functions $f$, and where $X$ follows the product measure $\mu^{\otimes d}$.
While hypercontractivity is classically studied for Gaussian, discrete, and uniform probability measures on hypercubes or hyperspheres \citep{bon70,nel73, bec75,bec92}.
We generalize it to function classes that have a certain graded structure.


\paragraph{Contribution.}
Our main result is to provide an upper bound for the number of samples $N$ such that
an $N$-point i.i.d. sample of random vectors contains the expectation in its convex hull,
i.e. the event \eqref{eq:rcv} occurs, with a reasonable probability.
Although the connection between the bound for $N$ and the hypercontractivity of
the given random vector/function class has implicitly been proven
in a preceding study \citep{hayakawa21a} in the form of Theorem~\ref{hykw-moment},
generic conditions for having a good hypercontractivity constant
and why the magnitude of required $N$ becomes reasonably small
in practice have not been established or understood.

In this paper, we address these questions by
\begin{itemize}
    \item
        extending the hypercontractivity for the Wiener chaos
        to what we call generalized random polynomials (Section~\ref{sec:3}) and
    \item
        showing that this extension naturally applies to important examples
        in numerical analysis including classical cubature,
        cubature on Wiener space, and kernel quadrature (Section~\ref{sec:app}).
\end{itemize}
We explain the intuition behind these points
by introducing Theorem~\ref{thm:informal} and Example~\ref{eg-main}:
\begin{thm}[informal]\label{thm:informal}
    Let $\mu$ be a probability measure on $\X$.
    Suppose we have a ``natural'' function class
    \[\F = \bigoplus_{d\ge 1}\bigcup_{m \ge0} \F_{d,m},\]
    where $\F_{d,m} $ denotes a set of functions from $\X^d$ to $\R$ of ``degree'' up to $m$. 
    Then, under some integrability assumptions, there exists for every $m$ a constant $C_m = C_m(\mu, \F)>0$  such that the following holds:
    \begin{quote}
        Let $d$ and $D$ be two positive integers and $\bm\phi=(\bm\phi_1,\ldots,\bm \phi_D):\X^d\to\R^D$ with $\bm \phi_1,\ldots,\bm\phi_D \in \F_{d, m}$.
        Then, for all integers $N\ge C_mD$, we have
        \[
            \P{\E{\bm\phi(X)} \in 
            \cv\{\bm\phi(X_1), \ldots, \bm\phi(X_N)\}}
            \ge \frac12,
        \]
        where $X, X_1, \ldots, X_N$ are i.i.d.~samples from the product measure $\mu^{\otimes d}$ on $\X^d$.
    \end{quote}
\end{thm}

\begin{eg}\label{eg-main}
Although the ``assumption'' in the above statement
is somewhat abstract, this applies to important examples as follows:
\begin{itemize}
    \item \textbf{Classical Cubature} \citep{str71}: $\mu$ is a probability measure with finite $m$ moments and $\F_{d,m}$ is the space of
    $d$-variate polynomials up to degree $m$ .
    \item \textbf{Cubature on Wiener space} \citep{lyo04}: $\mu$ is the Wiener measure and $\F_{d,m}$ is spanned by up to $m$-times iterated Ito--Stratonovich integrals.
    \item \textbf{Kernel quadrature} \citep{kar19,hayakawa21b}: $\mu$ is a probability measure on set $\X$ that carries a positive definite kernel $k$ and $\F_{d, m}$ is spanned by the eigenfunctions
    (down to some eigenvalue)
    of the integral operator
    $g\mapsto \int k^{\otimes d}(\cdot, x)g(x)
    \dd\mu^{\otimes d}(x)$,
    where $k^{\otimes d}$ is a tensor product kernel.
\end{itemize}
\end{eg}

\paragraph{Related work.}
If the measure $\mu$ has finite support,
the problem \eqref{eq:cubature} is also known as recombination. 
While in this case, the existence follows immediately from Caratheodory's theorem, the design of efficient algorithms to compute the cubature measure is more recent; we mention efficient deterministic algorithms
\citep{lit12,tch15,maa19}
and randomized speedups \citep{cos20}. 
For non-discrete measures, the majority of the cubature constructions are typically limited to algebraic approaches that cannot apply to general situations.
Related to our convex hull approach but different, is a line of research
aiming at constructing general cubature formulas
with positive weights
by using least-squares instead of the random convex hull approach
\citep{gla21,mig22}.
As their theory is on the positivity of the resulting cubature formula
given by solving a certain least squares problem,
requires more (or efficiently selected)
points than that needed for simply obtaining a positively weighted cubature.

Hypercontractivity is the key technical tool for our estimates and its use seems to be novel in the context of cubature resp.~random convex hull problems.
Somewhat related to the special case of kernel quadrature, \citep{mei21} proves a generalization error bound for kernel ridge regression with random features, however hypercontractivity is simply adopted as a technical assumption.
Further, for low-degree polynomials of a sequence of random variables,
\citet{kim00,sch12} give similar estimates on their higher order moments, but they mainly estimate the concentration of the moments,
and do not generally analyze the curtosis-type values appearing
in the hypercontractivity.

\paragraph{Outline.}
In Section \ref{sec:2},
we give briefly explain recent
results on random convex hulls,
and give some assertions that additionally follow from them.
In Section \ref{sec:3},
we introduce the Gaussian hypercontractivity and
show its generalization that is suitable for multivariate cubatures.
Section \ref{sec:app} gives some applications of Gaussian/generalized hypercontractibity
to random convex hulls with product structure,
including cubature on Wiener space and kernel quadrature.
Finally, we conclude the paper in Section \ref{sec:con}.
All the omitted proofs are given in Appendix \ref{sec:proofs}.

\section{Random Convex Hulls}\label{sec:2}
Our main interest is the probability of the even \eqref{eq:rcv} but it turns out to be more convenient to study a more general problem. 
Therefore we define
\begin{Definition}
Let $X$ be a $D$-dimensional random vector and $X_1, X_2, \ldots$ be iid copies of $X$.
For every integer $N > 0$ and $\theta\in\R^D$ define
\[
    p_{N, X}(\theta)\coloneqq\P{\theta \in \cv\{X_1, \ldots, X_N\}} \text{ and }
    N_X(\theta)\coloneqq\inf\{N\mid p_{N, X}(\theta)\ge1/2\}.
\]
\end{Definition}

Both of these quantities are classically studied for symmetric $X$
by \citet{wen62}, but more recently sharp inequalities for general $X$
\citep{wag01,hayakawa21a} as well
and calculations on the Gaussian case \citep{kab20} have been established.
Using this notation, our main interests is the choice $\theta =\E{\bm\phi(X)}$.
In the following two paragraphs we briefly discuss how bounds on $N_X(\theta)$ can be derived with previous approaches; in Section \ref{sec:3} we then discuss the approach via hypercontractivity.
\paragraph{Bounds via Tukey Depth.}
It turns out that a classical quantity from statistics, the so-called Tukey depth \citep{tuk75,rou99}, is closely related to the two quantities.
\begin{Definition}
The  Tukey depth of $\theta\in\R^D$ with respect to the distribution of $X$ is defined as
\begin{equation}
    \alpha_X(\theta) \coloneqq
    \inf_{c\in \R^D\setminus\{0\}}
    \P{c^\top(X - \theta) \le 0}.
    \label{eq:Tukey-depth}
\end{equation}
\end{Definition}
The relation between the above quantities is
\begin{thm}[\citep{hayakawa21a}]\label{hykw-tukey}
    Let $\theta\in\R^D$
    and $X$ be an arbitrary $D$-dimensional
    random vector.
    Then,
    we have
    \[
        \frac1{2\alpha_X(\theta)}
        \le N_X(\theta) \le \left\lceil 
        \frac{3D}{\alpha_X(\theta)}\right\rceil.
    \]
\end{thm}
The above can be used to provide a novel bound on $N_X(\E{X})$ for a general class of distributions called {\it log-concave},
\begin{prop}
    If $X$ is a 
    $D$-dimensional log-concave random vector,
    we have $N_X(\E{X})
    \le \lceil 3eD \rceil$.
\end{prop}

\paragraph{Bounds via Moments.}
Theorem \ref{hykw-tukey} gives
a good intuition behind the random convex hulls,
but computing the Tukey depth $\alpha_X(\theta)$ itself is in general a difficult task \citep{cue08,zuo19}.
In \citep{hayakawa21a} an alternative way to bound $N_X(\theta)$ is provided by using the Berry--Esseen theorem \citep{ber41,ess42,kor12}.
\begin{thm}[\citep{hayakawa21a}]\label{hykw-moment}
    Let $X$ be an arbitrary $D$-dimensional
    random vector
    with $\E{\lVert X \rVert^3} < \infty$.
    If a constant $K>0$ satisfies
    $   \lVert c^\top (X - \E{X}) \rVert_{L^3}
        \le K \lVert c^\top (X - \E{X}) \rVert_{L^2}$
    for all $c\in \R^D$,
    then we have \[N_X(\E{X}) \le 17(1 + 9K^6/4)D.\]
\end{thm}
This result still recovers a sharp bound
$N_X(\E{X}) = \ord{D}$ up to constant
for a Gaussian,
where we have detailed information about
the marginals.
The sort of inequality assumed in the statement
is also called Khintchin's inequality
(see e.g., \citep{kon14,esk18})
and there are known values of $B$
for a certain class of $X$ such as
a Rademacher vector.
Indeed, we can easily show the following estimate
under a clear independence structure:
\begin{prop}\label{prop:indep-easy}
    Let $X = (X_1, \ldots, X_D)^\top$
    be a $D$-dimensional random vector
    whose coordinates are independent and identically distributed.
    If $\E{X_1} = 0$
    and $\lVert X_1\rVert_{L^4} \le K\lVert X_1\rVert_{L^2}$
    holds for a constant $K>0$,
    then we have $\lVert c^\top X\rVert_{L^4}
    \le K\lVert c^\top X\rVert_{L^2}$ for all $c\in \R^D$.
\end{prop}

\section{Hypercontractivity}\label{sec:3}
The previous section provides bounds on $N_X$ but the assumptions--log-concavity or coordinate-wise independence--are too strong for many applications. 
We now develop an approach an appraoch via hypercontractivity; this results in bounds that apply under much less strict assumptions.
\paragraph{Hypercontractivity: the Gaussian case.}
It is instructive to briefly review the classical results for Gaussian measures by following \citet{GH-book} since we need several generalizations of this.

\begin{thm}[Wiener Chaos Decomposition]
Let $H$ be a Gaussian Hilbert space\footnote{A Gaussian Hilbert space is a closed linear subspace
of $L^2(\Omega, \mathcal{G}, \mathbb{P})$ whose elements all follow
Gaussian distributions.} and let $\sigma(H)$ be the $\sigma$-algebra generated by $H$. 
Then
    \[
        L^2(\Omega, \sigma(H), \mathbb{P})
        = \bigoplus_{n=0}^\infty H^{(n)},
    \]
where $H^{(n)}:= \overline{P_n(H)}
\cap P_{n-1}(H)^\perp$ with 
\[
    P_n(H)\coloneqq\{
        f(Y_1, \ldots, Y_m)
        \mid
        \text{$f$ is a polynomial of degree $\le m$},
        \ 
        Y_1,\ldots,Y_m\in H,
        \
        m < \infty
    \}
\]
with $P_{-1}(H) \coloneqq \{0\}$ and $\overline{P_n(H)}$ denotes the completion
in $L^2(\Omega, \F, \mathbb{P})$.
    \end{thm}
Hence, for each
$X\in L^2(\Omega, \sigma(H), \mathbb{P})$,
we have a unique decomposition
$X = \sum_{n=0}^\infty X_n$ such that $X_n\in H^{(n)}$.

\begin{thm}[Hypercontractivity, \cite{GH-book}, Theorem 5.8]\label{thm:hc}For $r \in [0,1]$ denote
\[
    T_r:L^2(\Omega, \sigma(H), \mathbb{P}) \to L^2(\Omega, \sigma(H), \mathbb{P}),\quad X \mapsto \sum_{n=0}^\infty r^n X_n.
\]
    If $p>2$ and $0 < r \le (p-1)^{-1/2}$,
    then we have
    \[
        \lVert T_r(X)\rVert_{L^p} \le \lVert X \rVert_{L^2}.
    \]
\end{thm}
From this, we have the following moment bound on
$\overline{P_n(H)}$,
which is also referred to as hypercontractivity, see for example \cite{nou10}.

\begin{thm}\label{thm:hc-moment}
    Let $n\ge0$ be an integer.
    For each $p > 2$,
    we have
    \[
        \lVert X\rVert_{L^p}
        \le (p-1)^{n/2}\lVert X \rVert_{L^2},
        \qquad X \in \overline{P_n(H)}.
    \]
\end{thm}
\begin{proof}
    Let $X = \sum_{m=0}^n X_m$ with $X_m\in H^{(m)}$.
    For $0 < r\le (p-1)^{-1/2}$,
    by Theorem \ref{thm:hc}, we have
    \[
        \lVert X\rVert_{L^p}^2
        = \left\lVert T_r\left(\sum_{m=0}^n r^{-m}X_m
        \right)\right\rVert_{L^p}^2
        \le
        \left\lVert \sum_{m=0}^n r^{-m}X_m
        \right\rVert_{L^2}^2
        =\sum_{m=0}^n r^{-2m}\lVert X_m\rVert_{L^2}^2
        \le r^{-2n}\lVert X\rVert_{L^2}^2.
    \]
    We obtain the conclusion by letting $r = (p-1)^{-1/2}$.
\end{proof}
We included the proof since we are going to generalize it in the next section.

\subsection{Hypercontractivity for Generalized Random Polynomials}\label{sec:gen-hc}
The phenomenon of hypercontractivity is not limited to the Gaussian setting.
Indeed, the hypercontractivity of operators
on the space of boolean functions (i.e., $\{-1, 1\}^n \to \R$)
associated with the uniform measure was established even before the Gaussian case \citep{bon70,sim72}.
Our focus is to obtain estimates analogous
to Theorem \ref{thm:hc-moment} when a graded class of test function is given; we refer to such a class as generalized random polynomials.
\paragraph{Generalized Random Polynomials.}
\begin{dfn}\it
    Under a probability space $(\Omega, \mathcal{G}, \mathbb{P})$,
    a triplet
    $G = (Y, Q, \lambda)$
    is called GRP if it satisfies the following conditions:
    \begin{itemize}
        \item $Y$ is a random variable
        taking values in a topological space $\X$.
        \item $Q = (Q_m)_{m=0}^\infty$
        is a nondecreasing
        sequence of linear spaces of
        $L^2(\mathbb{P}_Y)$-integrable functions $\X\to\R$.
        Namely, if we let $Q_m(Y):=\{f(Y)\mid f\in Q_m\}$,
        then each $Q_m$ is a linear subspace of
        $L^2(\mathbb{P})$, with
        $Q_0\subset Q_1\subset\cdots
        \subset L^2(\mathbb{P})$.
        We additionally assume $Q_0$ is the set of constant functions.
        \item $\lambda=(\lambda_m)_{m=0}^\infty$
        satisfies
        $1 = \lambda_0 > \lambda_1 \ge \lambda_2 \ge \cdots \ge 0$.
    \end{itemize}
    If $G$ is a GRP, we also define
    $\tilde{\deg}_G X :=
        \inf\{1/\lambda_m
            \mid 
            m\ge 0,\,
            X \in \overline{Q_m(Y)}
    \}$.
\end{dfn}

Intuitively,
each $Q_m$ is a generalization of degree-$m$ polynomials
and $\tilde{\deg}_G$ indicates the ``degree'' of
such functions (though $Y$ plays a role in the latter).
In the setting of actual polynomials like
Wiener chaos,
we can define $\lambda_m = b^{-m}$ for a certain $b>1$,
and then we have $\deg X = \log_b \tilde{\deg}_G X$
for the usual degree of $X$
as a random polynomial.

\begin{dfn}\it
Let $G=(Y,Q,\lambda)$ be a GRP.
We define
\[H_m(Y) \coloneqq
\overline{Q_m(Y)} \cap Q_{m-1}(Y)^\perp\]
in terms of $L^2(\mathbb{P})$ where $Q_{-1}(Y) \coloneqq \{0\}$ and  \[H_\infty \coloneqq L^2(\Omega, \sigma(Y), \p)
\cap \bigl(\bigcup_{m=0}^M Q_m(Y)\bigr)^\perp.\]
We refer 
\[
    L^2(\Omega, \sigma(Y), \p)
    = \left(\bigoplus_{m=0}^\infty H_m(Y)\right)
    \oplus H_\infty(Y)
\]
as the orthogonal decomposition associated with $G$.
\end{dfn}

\begin{dfn}\it
Let $G=(Y,Q,\lambda)$ be a GRP.
The operator $T(G)$ is defined as
\[
   T(G): L^2(\Omega, \sigma(Y), \p) \to  L^2(\Omega, \sigma(Y), \p), \quad X  \mapsto \sum_{m=0}^\infty \lambda_m X_m,
\]
where $(X_m)_{m \in \mathbb{N} \cup \infty}$ with $X_m \in H_m(Y)$ is the orthogonal decomposition of $X$ associated with the GRP $G$.
    We say that a GRP $G = (Y, Q, \lambda)$ is {\it $(2, p; s)$-hypercontractive}
    if 
    \[
        \lVert T(G)^s X \rVert_{L^p}
        \le \lVert X\rVert_{L^2},
        \qquad X \in L^2(\Omega, \sigma(Y), \p).
    \]
\end{dfn}
Thus, \[T(G)^s X = \sum_{m=0}^\infty \lambda_m^s X_m\text{ for }s>0\] and if $G$ is $(2, p; s)$-hypercontractive,
it is $(2, p; t)$-hypercontractive for all $t\ge s$ as $T(G)^{t-s}$
is a contraction in $L^2$.
The formulation of $G$ associated with ``degree'' concept
given by $\lambda$ then naturally extends to the
multivariate case.

\begin{dfn}\it
We call a set of $d$ GRPs,
$G^{(i)} = (Y^{(i)}, Q^{(i)}, \lambda^{(i)})$ for $i=1, \ldots, d$ independent, if the $Y^{(i)}$'s are independent random variables taking values in $\X^{(i)}$'s.
For $d$ independent GRPs, their product is a GRP 
$G=(Y, Q, \lambda)$ that is defined as follows
\begin{itemize}
    \item $Y = (Y^{(1)},\ldots,Y^{(d)})
    \in \X^{(1)} \times \cdots \times \X^{(d)}$.
    \item
    $\lambda_m$
    is the $(m+1)$-th largest value in the set
    $\left\{ \prod_{i=1}^d
    \lambda^{(i)}_{m_i} \lmid
    \lambda^{(i)}_{m_i} \in \lambda^{(i)},\, 
    i=1,\ldots,d
    \right\}$.
    \item
    $Q_m = \mathop\mathrm{span}
    \left\{
        f : (x_1, \ldots, x_d)
        \mapsto
        \prod_{i=1}^d f_i(x_i)
        \lmid
        f_i \in Q^{(i)}_{m_i},
        \, \prod_{i=1}^d \lambda^{(i)}_{m_i}
        \le \lambda_m
    \right\}$.
\end{itemize}
As $Q_m(Y) \subset L^2$ it follows from independence for each $m$ that $G = (Y, Q, \lambda)$ is indeed a GRP.
We also denote it by
$G = G^{(1)}\otimes\cdots\otimes G^{(d)}$.
\end{dfn}
\begin{Example}
Consider the case when $Q_m^{(i)}$ are degree-$m$ polynomials of $Y^{(i)}$ and $\lambda_m^{(i)} = t^m$ for some $t\in(0, 1)$ independent of $i$. 
This shows that the product GRP generalizes the multivariate random polynomials.
Also, when $Y^{(i)}$ are i.i.d. and $(Q^{(i)}, \lambda^{(i)})$ are the same for all $i=1,\ldots,d$, then we say $G^{(i)}$ are i.i.d. and we can particular write $G\simeq (G^{(1)})^{\otimes d}$.
\end{Example}
A straightforward generalization follows from the classical way of proving hypercontractivity. 
Nevertheless, it turns out to be very useful for treating multivariate hypercontractivity of our GRP setting.
\begin{thm}\label{key-prop}
    Let $r\in (0, 1]$ and $p>2$.
    If $d$ independent GRPs
    $G^{(1)}, \ldots, G^{(d)}$
    are all $(2, p; s)$-hypercontractive,
    then their product 
    $G = G^{(1)}\otimes\cdots\otimes G^{(d)}$
    is also $(2, p; s)$-hypercontracitve.
\end{thm}

\begin{rem}
    We only use the $(2, p; s)$-hypercontractivity
    in this paper, but
    we can also deduce the same results for
    the general $(q, p; s)$-hypercontractivity with
    $1\le q \le p < \infty$
    (for the operator norm of $L^q \to L^p$),
    analogous to e.g. \citet{GH-book}.
\end{rem}

The following is a parallel result of
Theorem \ref{thm:hc-moment}
and the proof is almost identical.
\begin{prop}\label{prop:gen-moment}
    Let $s>0$ and $p>2$.
    If $G$ is a GRP
    that is $(2, p; s)$-hypercontractive,
    then we have
    $\lVert X\rVert_{L^p}
    \le (\tilde{\deg}_GX)^s \lVert X\rVert_{L^2}$
    for all $X\in L^2$.
\end{prop}

\begin{rem}
    Although we have treated general GRPs $G=(Y, Q, \lambda)$
    in these propositions,
    we are basically only interested in the moment inequality
    for $X$ up to some ``degree''
    fixed beforehand
    (in the case of Wiener chaos,
    it suffices to treat $P_n(H)$ for some finite $n$
    to obtain Theorem \ref{thm:hc-moment}).
    Thus, our main interest is in ``finite''
    GRPs, satisfying
    $Q_n = Q_{n+1} = \cdots$ for some $n$,
    and their product in practice,
    which the might be better for readers to have in mind
    when reading the next proposition.
\end{rem}

We next show the following ``converse''
result
for the relation of the
hypercontractivity and moment estimate
for a (truncated) GRP when $p=4$.

\begin{prop}\label{prop:prod-hc}
    Let $G = (Y, Q, \lambda)$ be a GRP.
    Suppose there exists a $s>0$ such that
    \[
        \lVert X_m \rVert_{L^4} \le
        \lambda_m^{-s}\lVert X_m\rVert_{L^2},
        \qquad X_m\in H_m(Y)
    \]
    holds for all $m$.
    If $t>s$ satisfies
    \[\sum_{m\ge1}\lambda_m^{t-s}\le 1/\sqrt3\]
    and $\lambda_1^t\le1/2$,
    then $G$ is $(2,4;t)$-hypercontractive.
\end{prop}

By using this,
we can also prove the following
as a non-quantitative result.
\begin{thm}\label{thm:grp-main}
    Let $K>0$
    and $G$ be a GRP such that
    the space
    $\{X\in L^2 \mid \tilde{\deg}_G X \le K\}$
    is included in $L^4(\Omega,\F,\p)$
    and finite-dimensional.
    Then, there exists a constant
    $C=C(G, K)$
    such that for an arbitrary $d$,
    $
        \lVert X\rVert_{L^4}
        \le C\lVert X\rVert_{L^2}
    $
    holds
    if we have
    $\tilde{\deg}_{G^{\otimes d}}X \le K$.
\end{thm}

\section{Applications}\label{sec:app}
The generality of Proposition \ref{prop:prod-hc} and Theorem \ref{thm:grp-main} allows to quantify the number of samples resp.~probability of success of the random convex hull approach to the problem of cubature. 
Concretely, we give formal statements of Theorem \ref{thm:informal} for \begin{enumerate*}[label=(\roman*)]
  \item Classical Cubature,
  \item Cubature on Wiener Space,
  \item Kernel Quadrature.
\end{enumerate*} various cubature constructions. 
\subsection{Classical Polynomial Cubatures}
When the GRP $G$ are actual random polynomials, we recover the following result
\begin{cor}
    Let $m$ be a positive integer
    and $X^{(1)}, X^{(2)}, \ldots$ be i.i.d. real-valued random variables
    with $\E{\lvert X^{(1)}\rvert^{4m}}<\infty$.
    Then, there exists a constant $C_m>0$ such that
    \[
        \lVert f(X^{(1)}, \ldots, X^{(d)})\rVert_{L^4} \le
        C_m\lVert f(X^{(1)}, \ldots, X^{(d)})\rVert_{L^2}
    \]
    for any positive integer $d$ and any polynomial $f:\R^d\to\R$ with degree up to $m$.
\end{cor}
\begin{proof}
    By introducing a truncated GRP given by a random variable $X^{(1)}$,
    function spaces $Q_i$ of univariate polynomials up to degree $i$,
    and $\lambda_i=2^{-i}$ for $0\le i\le m$,
    we can apply Theorem \ref{thm:grp-main} to obtain the desired result.
\end{proof}

If we combine this with Theorem \ref{hykw-moment},
we obtain the following result for polynomial cubatures:
\begin{cor}
    Let $m\ge1$ be an integer and $X^{(1)}, X^{(2)},\ldots$
    be i.i.d. real-valued random variables
    with $\E{\lvert X^{(1)}\rvert^{4m}} < \infty$.
    Then, there exists a constant $C_m>0$ such that the following holds:
    \begin{quote}
        Let $d\ge1$ be an integer
        and $\bm\phi:\R^d\to\R^D$ be a $D$-dimensional vector-valued function
        such that each coordinate is given by a polynomial up to degree $m$.
        If we let $\bm{X}^{(1:d)}_1, \bm{X}^{(1:d)}_2, \ldots$
        be independent copies of $\bm{X}^{(1:d)} = (X^{(1)}, \ldots, X^{(d)})$,
        we have
        \[
            \P{\E{\bm\phi(\bm{X}^{(1:d)})}\in\cv\{\bm\phi(\bm{X}^{(1:d)}_1),
            \ldots, \bm\phi(\bm{X}^{(1:d)}_N)\}} \ge\frac12
        \]
        for all integers $N\ge C_mD$.
    \end{quote}
\end{cor}

\subsection{Cubature on Wiener Space}\label{sec:cow}

Cubature on Wiener space \citep{lyo04}
is a weak approximation scheme for stochastic differential equations; at the hear of it is the construction of a finite measure on pathspace, such that the expectation of their first $m$-times iterated integrals matches those of Brownian motion. 
Some algebraic constructions are known for degree $m = 3, 5$ \citep{lyo04}
as well as $m=7$ \citep{nin21}. 
The random convex hull approach applies in principle for any $m$, however, a caveate is that the discretization of paths becomes an issue in particular for high values of $m$; some experimental results are available in \citep{hayakawa-CoW} for constructing them by using random samples of piecewise linear approximations of Brownian motion.
In this section, we use hypercontractivity to estimate the number of samples needed in this approach to cubature via sampling.

\paragraph{Random Convex Hulls of Iterated Integrals.}
For a bounded-variation (BV) path $x=(x^0, \ldots, x^d):[0, 1]\to\R^{d+1}$
and a $d$-dimensional standard Brownian motion $B = (B^1, \ldots, B^d)$ with $B^0_t:=t$,
we define the iterated integrals as
\[
    I^\alpha(x):=\int_{0<t_1<\cdots<t_k<1}\dd x^{\alpha_1}_{t_1}\cdots\dd x^{\alpha_k}_{t_k},
    \qquad
    I^\alpha(B):=\int_{0<t_1<\cdots<t_k<1}\circ\dd B^{\alpha_1}_{t_1}\cdots\circ\!\dd B^{\alpha_k}_{t_k},
\]
where the latter is given by the Stratonovich stochastic integral.
Then, a degree $m$ cubature formula on Wiener space
for $d$-dimensional Brownian motion is
a set of BV paths $x_1,\ldots,x_n :[0, 1]\to\R^{d+1}$
and convex weights $w_1,\ldots,w_n$ such that
$\sum_{i=1}^n w_i I^\alpha(x_i) = \E{I^\alpha(B)}$
for all multi-indices $\alpha = (\alpha_1, \ldots, \alpha_k)\in\bigcup_{\ell\ge1}\{0, 1, \ldots, d\}^\ell$
with $\lVert\alpha\rVert:=k+\lvert\{j\mid \alpha_j=0\}\rvert\le m$.

Indeed, if we consider the Gaussian Hilbert space
given by
\[
    H:=\left\{
        \sum_{i=1}^d\int_0^1 f_i(t)\dd B^i_t
        \lmid
        f_1,\ldots,f_d\in L^2([0, 1])
    \right\},
\]
the iterated integral $I^\alpha(B)$
with $\lVert\alpha\rVert\le m$
is in the $m$-th Wiener chaos $\overline{P_m(H)}$
(see Section \ref{sec:3})
as it can be expressed as
a limit of polynomials of increments of $B$.
We thus have the hypercontructivity given in Theorem \ref{thm:hc-moment}
and the following assertion:
\begin{cor}
    Let $d, m\ge1$ be integers and $B$ be a $d$-dimensional Brownian motion.
    Then, for an arbitrary linear combination
    $X = \sum_{\lVert\alpha\rVert\le m}c_\alpha I^\alpha(B)$
    with $c_\alpha\in\R$,
    we have $\lVert X\rVert_{L^3} \le 2^{m/2}\lVert X\rVert_{L^2}$.
\end{cor}
As the bound is independent of the dimension $d$ of the underlying Brownian motion,
we have the following version of Theorem \ref{thm:informal}
by combining it with Theorem \ref{hykw-moment}
as follows:
\begin{cor}
    Let $d, m\ge1$ be integers and $B, B_1, B_2,\ldots$ be independent standard $d$-dimensional Brownian motions.
    Then, if $\bm\phi(B)$ is a $D$-dimensional random vector such that each coordinate
    is given by a linear combination of $(I^\alpha(B))_{\lVert\alpha\rVert\le m}$,
    then we have
    \[
        \P{\E{\bm\phi(B)}\in\cv\{\bm\phi(B_1), \ldots, \bm\phi(B_N)\}} \ge\frac12
    \]
    for all integers $N\ge 17(1+18\cdot8^{m-1})D$.
\end{cor}
The above allows to choose the number of candidate paths that need to be sampled. 
However, as mentioned above, one challenge that is specific to cubature on pathspace is that one cannot sample a true Brownian trajectory which leads to an additional discretization error. 
However, we conjecture that the number of random samples divided by $D$ and the number of time partitions for piecewise linear approximation in constructing cubature on Wiener space can be independent of the underlying dimension $d$.
\begin{rem}
One can also apply these estimates to fractional Brownian motion \citep{pas16},
though we also need to obtain the exact expectations
of iterated integrals of fractional Brownian motion
(we can find some results on the Ito-type iterated integrals
without the time integral by $B^0_t=t$ in the literature \citep[][Theorem 31]{bau07}).
\end{rem}

\subsection{Kernel Quadrature for Product Measures}\label{sec:app-rkhs}
Let $\X$ be a topological space
and $k:\X\times\X\to\R$ be a positive definite kernel
with the reproducing kernel Hilbert space (RKHS) $\Hil_k$ \citep{ber04}.
A {\it kernel quadrature} for a random variable $X$
or equivalently a Borel probability measure $\mu$ (i.e., $X\sim\mu$) on $\X$ is a cubature formula for $(\Hil_k, \mu)$; that is, a set of points $x_i \in \X$ and
weights $w_i \in \R$ such that $\tilde \mu_n = \sum w_i \delta_{x_i}$ minimizes 
worst-case error
\begin{equation}
    \wce(\tilde{\mu}_n; \Hil_k, \mu)\coloneqq
    \sup_{\lVert f\rVert_{\Hil_k} \le 1}\left\lvert
        \E{f(X)} - \sum_{i=1}^n w_i f(x_i)
    \right\rvert,
    \label{eq:wce}
\end{equation}
which we might just denote by $\wce(\tilde{\mu}_n)$,
has been widely studied from
the viewpoint of optimization \citep{che10,bac12,hus12}
as well as sampling \citep{bac17,bel19,hayakawa21b}.
\paragraph{Tensor Product Kernels.}
When there are $d$ pairs of space and kernel
$(\X_1, k_1), \ldots, (\X_d, k_d)$,
the {\it tensor product kernel}
on the product space $\X_1\times\cdots\times\X_d$
is defined as
\[
    (k_1\otimes\cdots\otimes k_d)(x, y) := \prod_{i=1}^d k_i(x_i, y_i),
    \quad 
    x = (x_1, \ldots, x_d),
    y = (y_1, \ldots, y_d)\in \X_1\times\cdots\times\X_d.
\]
This is indeed the reproducing kernel of
the tensor product $\Hil_{k_1}\otimes\cdots\otimes\Hil_{k_d}$
in terms of RKHS \citep{ber04}.
The most important example of this construction is when the underlying $d$ kernels are the same,
$k^{\otimes d}=k\otimes\cdots\otimes k$.
Given a probability measure
$\mu$ in the (conceptually univariate) space
$\X$, constructing a kernel quadrature for $\mu^{\otimes d}$ with respect to $k^{\otimes d}$ is a natural multivariate extension of kernel quadrature that is widely studied in the literature \citep{oha91,kan16,bac17,kar19}, and corresponds to high-dimensional QMCs \citep{dic13}.
\paragraph{Mercer Expansions and Quadrature.}
The convergence rate of $\wce(\tilde{\mu}_n)$
is typically described by using
the Mercer expansion:
\begin{equation}
    k(x, y) = \sum_{\ell=1}^\infty \sigma_{\ell} e_{\ell}(x)e_{\ell}(y),
    \label{eq:mercer-exp}
\end{equation}
where $(\sigma_{\ell}, e_{\ell})_{\ell=1}^\infty$
are eigenpairs of
the integral operator
$\K: f \mapsto \int_\X k(\cdot, y)f(y)\dd\mu(y)$
in $L^2(\mu)$
with $\sigma_1\ge\sigma_2\ge\cdots\ge 0$
and $\lVert e_{\ell}\rVert_{L^2(\mu)} = 1$.

\begin{assp}\label{assp:ker}
    The kernel $k$ satisfies that
    the expansion \eqref{eq:mercer-exp} converges
    pointwise,
    $\sum_{\ell=1}^\infty\sigma_{\ell}<\infty$,
    and $(\sqrt{\sigma_{\ell}}e_{\ell})_{\ell=1}^\infty$ is an
    orthonormal basis of $\Hil_k$.
\end{assp}
Mild conditions already imply that Assumption \ref{assp:ker} applies, e.g., 
$\supp \mu = \X$, $k$ is continuous, and $x\mapsto k(x, x)$ is in $L^1(\mu)$ is sufficent, see \citep{ste12}.
Under this assumption, an $n$-point kernel quadrature rule that exactly integrates
the first $n-1$ eigenfunctions
satisfies the following theoretical guarantee:
\begin{prop}[\citep{hayakawa21b}]
    Under Assumption \ref{assp:ker},
    let $\tilde{\mu}_n = (w_i, x_i)_{i=1}^n$ be a kernel quadrature
    with convex weights satisfying
    $\int_\X e_{\ell}(x)\dd\mu(x) = \sum_{i=1}^nw_ie_{\ell}(x_i)$
    for each $\ell = 1, \ldots, n-1$.
    Then, by letting 
    $r_n(x):=\sum_{m=n}^\infty\sigma_me_m(x)^2$,
    we have
    $\wce(\tilde{\mu}_n)^2 \le 4\sup_{x\in\X}r_n(x)$.
\end{prop}
We have more favorable bounds on $\wce(\tilde{\mu}_n)$
by assuming more, but the important fact here
is that the event \eqref{eq:rcv}
for a vector-valued $\bm\phi$ given by eigenfunctions
$e_1,\ldots,e_{n-1}$ enables us to construct
an interesting numerical scheme.
A similar approach, specialized to a Gaussian kernel over a Gaussian measure
can be found in \citep{kar19}.
As the construction of such $\tilde{\mu}_n$
for general $k$ and $\mu$ relies on random sampling,
we want to estimate
$N_{\bm\phi(X)}(\E{\bm\phi(X)})$
for $X\sim\mu$ and
$\bm\phi = (e_1, \ldots, e_{n-1})$.
\paragraph{From RKHS to GRP.}
To make it compatible with the framework of GRPs
introduced in the previous section,
we further assume the following condition,
which ensures that the kernel is in an appropriate scaling.
\begin{asspp}\label{asp:grp}
    The kernel $k$ satisfies
    Assumption \ref{assp:ker},
    $\sigma_1 \le 1$,
    and the strict inequality
    $\sigma_\ell < 1$ holds if $e_\ell\in L^2(\mu)$
    is not constant.
\end{asspp}

Under Assumption \ref{asp:grp},
we can naturally define a GRP $G = (Y, Q, \lambda)$
with
$Y$ following $\mu$,
$Q_m = \mathop\mathrm{span}\{1, e_1, \ldots, e_m\}$
and $\lambda_m = \sigma_m$ for $m\ge1$.
Note that it violates the condition $\lambda_1 < 1$
if $\sigma_1 = 1$ and $e_1$ is constant, but in that case
we can simply decrement all the indices of $(Q_m, \lambda_m)$ by one.
We call it the {\it natural GRP} for $k$ and $\mu$
and denote it by $G = G_{k, \mu}$.

\begin{rem}\label{rem:scaling}
    The scaling given in Assumption \ref{asp:grp}
    is essential to the hypercontractivity
    under the framework of tensor product kernels
    when considering ``eigenspace down to some eigenvalue.''
    Indeed, if $\sigma_\ell \ge 1$
    for some nonconstant eigenfunction $e_\ell$,
    we have, for $p>2$,
    \[
        \frac{\lVert e_\ell^{\otimes d} \rVert_{L^p(\mu^{\otimes d})}}{
            \lVert e_\ell^{\otimes d} \rVert_{L^2(\mu^{\otimes d})}
        }
        =
        \left(
            \frac{\lVert e_\ell \rVert_{L^p(\mu)}}{
            \lVert e_\ell \rVert_{L^2(\mu)}
        }
        \right)^d
    \]
    which increases exponentially as $d$ grows,
    whereas the corresponding eigenvalue is lower bounded by $1$.
    So the hypercontractivity in our sense never gets satisfied if $\sigma_\ell\ge 1$
    for a nonconstant $e_\ell$.
\end{rem}

The following
statement, written without GRPs,
is what we can prove by using the hypercontractivity of GRPs.
\begin{prop}
    Let $k$ satisfy Assumption \ref{asp:grp}
    and $Y_1, Y_2, \ldots$ independently follow $\mu$.
    For each $\delta > 0$,
    define a set of random variables as
    \[
        S(\delta):=\mathop\mathrm{span}
        (\{1\}\cup
        \{ e_{\ell_1}(Y_{m_1})\cdots e_{\ell_k}(Y_{m_k})
        \mid k\ge1,\, 
        m_1 < \cdots < m_k,\,
        \sigma_{\ell_1}\cdots\sigma_{\ell_k}\ge \delta\}).
    \]
    Then, if
    $\lVert e_\ell(Y_1)\rVert_{L^4}<\infty$ holds
    for all $\ell$ with $\sigma_\ell\ge\delta$,
    then
    there is a constant $C = C(\delta)>0$ such that
    $\lVert X\rVert_{L^4}\le C \lVert X\rVert_{L^2}$
    for all $X\in S(\delta)$.
\end{prop}
\begin{proof}
    The finiteness of the dimension of eigenspace for $Y_1$, i.e,
    the finiteness of $\ell$ satisfying $\sigma_\ell\ge\delta$
    follows from $\sum_{\ell=1}^\infty\sigma_\ell<\infty$ in Assumption \ref{assp:ker}.
    Thus, Theorem \ref{thm:grp-main} gives the conclusion.
\end{proof}
This assertion, of course, includes a hypercontractivity statement
for an eigenspace of $k^{\otimes d}$ and $\mu^{\otimes d}$
for a fixed $d$,
but we can go further to a quantitative statement
by imposing another assumption.

\begin{assp}\label{assp:ortho}
    The kernel $k$ can be written as $k = 1 + k_0$,
    where $k_0:\X\times\X\to\R$ is a positive definite kernel
    satisfying $\int_\X k_0(x, y)\dd\mu(y) = 0$ for ($\mu$-almost) all
    $x\in\X$.
\end{assp}
Under Assumption \ref{assp:ker},
this is simply equivalent to $e_1$ being constant.
This assumption might seem artificial,
but naturally arises in the following situations:
\begin{itemize}
    \item[(a)]
        $\X$ is a compact group
        and $\mu$ is its Haar measure.
        $k$ is a positive definite kernel
        given as $k(x, y) = g(x^{-1}y)$,
        where $g:\X\to\R_{\ge0}$ and $\int_\X g(x)\dd\mu(x) = 1$.
    \item[(b)]
        $k_0$ is a kernel called Stein kernel \citep{oat17,ana21}
        with appropriate scaling.
\end{itemize}
One theoretically sufficient condition
for these assumptions
can be described as follows:
\begin{prop}\label{prop:asp-suff}
    Let $\X$ be compact metrizable and path-connected,
    $\supp\mu = \X$,
    and $k$ be continuous and nonnegative.
    If $\int_\X k(x, y)\dd\mu(y)=1$ holds for all $x\in \X$,
    Assumption \ref{asp:grp} and \ref{assp:ortho} hold.
\end{prop}

From this proposition, for instance,
an appropriately scaled exponential/Gaussian kernel
over the $n$-sphere with the uniform measure satisfies
Assumption \ref{asp:grp} and \ref{assp:ortho}.

Under these two assumptions,
the operator $T(G_{k, \mu})$
in terms of GRPs corresponds
to the integral operator
$\mathcal{K} : f\mapsto
\int_\X k(\cdot, y)f(y)\dd\mu(y)$,
so the situation becomes even simpler.
We can directly apply Proposition \ref{prop:prod-hc}
by replacing $\lambda$'s with $\sigma$'s,
but we also have the following sufficient conditions
for the hypercontractivity without explicitly
using the eigenvalue sequence.
In the following, $\lVert \mathcal{K}_0\rVert:=\sigma_2 < 1$
is the operator norm
of $\mathcal{K}_0:f\mapsto\int_\X
k_0(\cdot, y)f(y)\dd\mu(y)$ on $L^2(\mu)$,
and $\tr(\mathcal{K}_0):=\int_\X k_0(x, x)\dd\mu(x)$.
We may have the following quantitative condition for hypercontractivity.

\begin{prop}\label{prop:ker-hc}
    Let $k = 1 + k_0$ satisfy Assumption \ref{asp:grp} and \ref{assp:ortho}.
    When $\lVert\mathcal{K}_0\rVert > 0$,
    if $r,s\ge1$ satisfy
    \[
        \lVert \mathcal{K}_0\rVert^{-(r+s)} \ge 2,
        \quad
        \lVert \mathcal{K}_0\rVert^{-(r-1)} \ge \sqrt{3}\tr(\mathcal{K}_0),
        \quad
        \lVert \mathcal{K}_0\rVert^{-(s-1)} \ge \lVert k_0\rVert_{L^4(\mu\otimes\mu)},
    \]
    then $G_{k, \mu}$ is $(2, 4; r+s)$-hypercontractive.
    In particular, if we have
    $\sup_{x\in\X}\lvert k_0(x, x)\rvert \le 1/\sqrt{3}$,
    then $G_{k, \mu}$ is $(2, 4; 2)$-hypercontractive.
\end{prop}

\begin{Example}[Periodic Sobolev spaces over the torus.]
Following \citet{bac17}, we consider periodic kernels over $[0, 1]$.
Therefore let $\X = [0, 1]$, $\mu$ be the uniform distribution on $\X$,
and define
\begin{equation}
    k_{r, \delta}(x, y) = 1 + \delta\cdot\frac{(-1)^{r-1}(2\pi)^{2r}}{(2r)!}
    B_{2r}(\lvert x - y \rvert)
    \label{eq:sob}
\end{equation}
for each positive integer $s$
and $\delta\in(0, 1)$,
where $B_{2r}$ is the $2r$-th Bernoulli polynomial \citep{wah90}.
$\delta=1$ is assumed in the original definition, but
it violates Assumption \ref{asp:grp} (see also Remark \ref{rem:scaling}).
Albeit this slight modification,
the kernel $k_{r,\delta}$ gives an equivalent norm
to the periodic Sobolev space in the literature.
For $\delta\in(0, 1)$,
$k_{r, \delta}$ satisfies Assumption \ref{asp:grp} and \ref{assp:ortho}.
The eigenvalues and eigenfunctions with respect to the uniform measure
are known \citep{bac17}; the eigenvalues are:
$1$ for the constant function,
and $\delta m^{-2r}$ for $c_m(\cdot):=\sqrt2\cos(2\pi m\, \cdot)$
and $s_m(\cdot):=\sqrt2\sin(2\pi m\, \cdot)$ for
$m\ge 1,2,\ldots$.
We now apply Proposition \ref{prop:prod-hc} with (for sake of concreteness) $\delta = 1/3$.
This gives $\lVert c_m\rVert_{L^4(\mu)}
= \lVert s_m\rVert_{L^4(\mu)} =(3/2)^{1/4}$.
Thus, to satisfy the condition of Proposition \ref{prop:prod-hc},
it suffices for $s < t$ to satisfy
$3^s \ge (3/2)^{1/4}$, $\delta^{t-s}\zeta(2r(t-s))$,
$3^t \ge 2$,
where $\zeta$ is Riemann's zeta function.
Hence a simple numerical sufficient condition for this is $s=0.1$ and $t=1.1$ for $r=1$,
and $t=\log_32\le 0.631$ for $r\ge2$,
which can be derived by letting $2r(t-s)\ge 2$.
To sum up, in the case $r\ge2$,
we only need $\ord{\lambda^{-0.631}D}$ times of sampling
for meeting \eqref{eq:rcv} with probability over a half,
if $X\sim\mu^{\otimes d}$
and each coordinate of $\bm\phi:\X^d\to\R^D$
is in the eigenspace of the eigenvalue $\lambda$.
\end{Example}
\section{Concluding remarks}\label{sec:con}
We investigated the number of samples needed for the expectation vector to be contained in their convex hull from the viewpoint of product/graded structure.
We showed that the fact that we empirically only need $\ord{D}$ times
of sampling for the $D$-dimensional random vector
in practical examples
can partially be explained by the
hypercontractivity in the Gaussian case
as well as the generalized situation including random polynomials and
product kernels.
There are also interesting questions for further research; for example, although in the asymptotic $d\to\infty$ we established that the required number of sampling divided by $D$ is independent of $d$, the constants are larger than what purely empirical estimates given in \citep{hayakawa-MCCC,hayakawa21b} (where
$10D$ is sufficient in practice).
Another direction, is the case of cubature of Wiener space, as one cannot actually sample from Brownian motion and discretization errors propage to higher order $m$; an promising research direction could be to study ``approximate sampling'' or consider unbiased simulations \citep{hen17} for the iterated integrals.

\section*{Acknowlegments}
Harald Oberhauser and Terry Lyons are supported by
the DataSıg Program [EP/S026347/1], the Alan Turing Institute [EP/N510129/1], the Oxford-Man Institute, and the CIMDA collaboration by City University Hong Kong and the University of Oxford.

\bibliography{cite}
\bibliographystyle{abbrvnat}

\appendix

\section{Log-concave distributions}\label{sec:log-concave}
A function $f:\R^d \to \R_{\ge0}$
is called {\it log-concave}
if it satisfies
\[
    f(tx + (1-t)y) \ge f(x)^t f(y)^{1-t}
\]
for all $x,y\in\R^d$ and $t\in[0, 1]$.
A probability distribution with
a log-concave density is also called
log-concave,
and this class includes the multivariate Gaussian/exponential/Wishart distributions,
the uniform distribution over a convex domain,
and many more univariate common distributions
\citep{an97,boy04}.
For the log-concave random vectors,
the following result is known:

\begin{thm}[\citep{cap91}]
    If $X$ is a $d$-dimensional random vector
    with log-concave density,
    then we have $\alpha_X(\E{X}) \ge 1/e$.
\end{thm}

Here,
$\alpha_X (\E{X})$ is the Tukey depth of $\E{X}$
with respect to the ditrtibution of $X$
which is defined as \eqref{eq:Tukey-depth}.
The case when $X$ is uniform over a convex set is
proven in \citet{gru60},
and \citet[][Section 5]{lov07}
gives simpler proofs than the original result
in \citet{cap91}.

\section{Proofs}\label{sec:proofs}
\subsection{Proof of Proposition \ref{prop:indep-easy}}
\begin{proof}
    It suffices to consider the case
    $\lVert X_1\rVert_{L^4} < \infty$.
    If we write $c = (c_1, \ldots, c_D)^\top$,
    then by using independence we have
    \begin{align*}
        \lVert c^\top X \rVert_{L^4}^4
        &= \E{(c^\top X)^4}
        = \sum_{i=1}^Dc_i^4 \E{X_i^4}
         + \sum_{1\le i < j\le D} c_i^2c_j^2\E{X_i^2}\E{X_j^2}\\
        &\le K^4\sum_{i=1}^Dc_i^4 \E{X_i^2}^2
         + \sum_{1\le i < j\le D} c_i^2c_j^2\E{X_i^2}\E{X_j^2}\\
        &\le K^4 \left(\sum_{i=1}^Dc_i^2\E{X_i^2}\right)^2
        \le K^4 \E{(c^\top X)^2}^2,
    \end{align*}
    as we clearly have $K\ge1$ (or $X = 0$ almost surely).
\end{proof}

\subsection{Proof of Theorem \ref{key-prop}}
\begin{proof}
    We give the proof by generalizing
    the proof of Lemma 5.3 in 
    \citet{GH-book}.
    
    It suffices to prove the statement for $d=2$,
    as the product of GRPs is associative.
    Let $G^{(i)}=(Y^{(i)}, Q^{(i)}, \lambda^{(i)})$
    for $i=1,2$
    be independent GRPs.
    Let $H_m^{(i)}(Y^{(i)}):=
    \overline{Q_m^{(i)}(Y^{(i)})}
    \cap Q_{m-1}^{(i)}(Y^{(i)})^\perp$
    for $i=1, 2$.
    If we denote the product by
    $G = G^{(1)}\otimes G^{(2)}$.
    Then,
    for a random variable
    $X = \sum_{\ell, m} X_{\ell,m}$
    with $X_{\ell,m}\in H_\ell^{(1)}\otimes H_m^{(2)}$,
    the operator $T(G)$ acts as
    \[
        T(G)X = \sum_{\ell,m}
        \lambda^{(1)}_\ell \lambda^{(2)}_m X_{\ell, m}.
    \]
    If each $X_{\ell, m}$ can be written
    as a finite sum
    $X_{\ell, m} = \sum_k
    X^{(1)}_{k,\ell,m}X^{(2)}_{k, \ell, m}$
    with $X^{(1)}_{k,\ell,m}\in H_\ell^{(1)}(Y^{(1)})$
    and $X^{(2)}_{k,\ell,m}\in H_m^{(2)}(Y^{(2)})$,
    then by using Minkowski's integral inequality
    \citep{inequalities} and
    the $(2, p; s)$-hypercontractivity
    of $G^{(1)}$ and $G^{(2)}$,
    we have
    \begin{align*}
        \lVert T(G)^sX\rVert_{L^p}
        &= \mathbb{E}_{Y^{(1)}}\!\!
        \left[\mathbb{E}_{Y^{(2)}}\!\!\left[
            \left\lvert
                \sum_{\ell, m}(\lambda_\ell^{(1)}
                \lambda_m^{(2)})^s X_{\ell, m}
            \right\rvert^p
        \right]\right]^{1/p}\\
        &= \mathbb{E}_{Y^{(1)}}\!\!
        \left[\mathbb{E}_{Y^{(2)}}\!\!\left[
            \left\lvert
                \sum_{k, \ell, m}
                (\lambda_\ell^{(1)})^s
                X^{(1)}_{k,\ell,m}
                (\lambda_m^{(2)})^s X^{(2)}_{k,\ell, m}
            \right\rvert^p
        \right]\right]^{1/p}\\
        &\le \mathbb{E}_{Y^{(1)}}\!\!
        \left[\mathbb{E}_{Y^{(2)}}\!\!\left[
            \left\lvert
                \sum_{k, \ell, m}
                (\lambda_\ell^{(1)})^s
                X^{(1)}_{k,\ell,m}
                X^{(2)}_{k,\ell, m}
            \right\rvert^2
        \right]^{p/2}\right]^{1/p}
        \tag{by $G^{(2)}$}\\
        &
        \le \mathbb{E}_{Y^{(2)}}\!\!
        \left[\mathbb{E}_{Y^{(1)}}\!\!\left[
            \left\lvert
                \sum_{k, \ell, m}
                (\lambda_\ell^{(1)})^s
                X^{(1)}_{k,\ell,m}
                X^{(2)}_{k,\ell, m}
            \right\rvert^p
        \right]^{2/p}\right]^{1/2}
        \tag{by Minkowski}\\
        &\le
        \mathbb{E}_{Y^{(2)}}\!\!
        \left[\mathbb{E}_{Y^{(1)}}\!\!\left[
            \left\lvert
                \sum_{k, \ell, m}
                X^{(1)}_{k,\ell,m}
                X^{(2)}_{k,\ell, m}
            \right\rvert^2
        \right]\right]^{1/2}
        \tag{by $G^{(1)}$} = \lVert X\rVert_{L^2}.
    \end{align*}
    The general case follows from the limit argument.
\end{proof}

\subsection{Proof of Proposition \ref{prop:gen-moment}}
\begin{proof}
    Let $G = (Y, Q, \lambda)$.
    Suppose $\tilde{\deg}_GX < \infty$ and
    let $n$ be the minimum integer
    satisfying $X\in \overline{Q_n(Y)}$.
    Then, by decomposing $X = \sum_{m = 0}^n X_m$
    with $X_m \in H_m(Y)$,
    we obtain
    \begin{equation*}
        \lVert X\rVert_{L^p}
        = \left\lVert
            T(G)^s \sum_{m=0}^n
            \lambda_m^{-s}X_m
        \right\rVert_{L^p}
        \le \left\lVert
            \sum_{m=0}^n
            \lambda_m^{-s}X_m
        \right\rVert_{L^2}
        \le \lambda_m^{-s}\lVert X\rVert_{L^2},
    \end{equation*}
    where we have used the $(2,p; s)$-hypercontractivity
    in the second inequality.
\end{proof}

\subsection{Proof of Proposition \ref{prop:prod-hc}}
\begin{proof}
    It suffices to consider $X$
    having the decomposition
    $X = \sum_{m} X_m$ with $X_m \in H_m(Y)$.
    Recall that we have assumed that
    $Q_0$ is the space of constant functions,
    so $X_0$ is a constant.
    First,
    we consider the case $X_0=0$.
    In this case,
    for $t>s$,
    we have
    \begin{align*}
        \lVert T(G)^t X \rVert_{L^4}^2
        &=\left\lVert
            \sum_{m\ge1} \lambda_m^t X_m
        \right\rVert_{L^4}^2
        \le \left(\sum_{m\ge1} \lambda_m^{t-s}
        \lambda_m^s
        \lVert X_m \rVert_{L^4}\right)^2\\
        &\le \left(\sum_{m\ge1}
        \lambda_m^{t-s} \lVert X_m\rVert_{L^2}\right)^2
        \le \left(\sum_{m\ge1} \lambda_m^{2(t-s)}\right)
        \lVert X \rVert_{L^2}^2
        \tag{Cauchy--Schwarz}.
    \end{align*}
    Therefore, when $\sum_{m\ge1}\lambda_m^{2(t-s)}\le 1/\sqrt3$
    we have
    \begin{equation}
        \lVert T(G)^t X \rVert_{L^4}
        \le 3^{-1/4}\lVert X\rVert_{L^2}
        \label{eq:centered-hc}
    \end{equation}
    for all $X$ satisfying $X_0=0$.
    
    In the case $X_0\ne0$,
    we can assume $X_0 = 1$ without
    loss of generality.
    Let
    $W = X - 1$ and
    $Z = T(G)^tW = T(G)^tX - 1$.
    Note that $\E{W}=\E{Z}=0$ holds
    by the orthogonality.
    We can explicitly expand the $L^4$ norm as follows:
    \begin{align*}
        \lVert T(G)^t X\rVert_{L^4}^4
        &= 1 + 6\E{Z^2} + 4\E{Z^3} + \E{Z^4}\\
        &\le 1 + 8\E{Z^2} + 3\E{Z^4}.
        \tag{AM--GM}
    \end{align*}
    We also have
    \[
        \lVert X\rVert_{L^2}^4
        = \E{(1 + W)^2}^2
        = (1+\E{W^2})^2
        = 1 + 2\E{W^2} + \E{W^2}^2.
    \]
    So it suffices to show $4\E{Z^2}\le \E{W^2}$ and
    $3\E{Z^4} \le \E{W^2}^2$, but the latter
    immediately follows from \eqref{eq:centered-hc}.
    The former holds when $\lambda_1^t\le 1/2$:
    \[
        \E{Z^2}
        = \sum_{m\ge1} \lambda_m^{2t}\E{X_m^2}
        \le \lambda_1^{2t} \E{W^2}.
    \]
    Therefore, we have completed the proof.
\end{proof}

\subsection{Proof of Theorem \ref{thm:grp-main}}
\begin{proof}
    Let $G = (Y, Q, \lambda)$
    and $\X$ be the space in which $Y$ takes values.
    By truncating $Q$ and $\lambda$
    (i.e., ignoring $Q_m$ with $1/\lambda_m > K$),
    we can assume that $Q(Y) =
    \{X\in L^2\mid \tilde{\deg}_GX\le K\}$.
    Then, as $\dim Q<\infty$,
    we can take a vector-valued measurable function
    \[
        \bm\phi = (\phi_1, \ldots, \phi_N)^\top
        : \X \to \R^N
    \]
    such that $(\phi_i(Y))_{i=1}^N$
    is an orthonormal basis of $Q(Y)$.
    Then, we have
    \[
        \sup_{X\in Q(Y)\setminus\{0\}}
        \frac{\lVert X\rVert_{L^4}}{\lVert X\rVert_{L^2}}
        = \sup_{c\in \R^N\setminus\{0\}}
        \frac{\lVert c^\top\bm{\phi}(Y)
        \rVert_{L^4}}{\lVert c^\top\bm{\phi}(Y)
        \rVert_{L^2}}
        = \sup_{c\in\R^N,\,
        \lVert c\rVert = 1} \lVert
            c^\top\bm\phi(Y)
        \rVert_{L^4}
        <\infty,
    \]
    where the right-hand side is the supremum of
    a continuous functions over a compact domain,
    and so is indeed finite.
    Hence, we can apply Proposition \ref{prop:prod-hc},
    and there exists a constant $s>0$ such that
    \[
        \left\lVert T(G)^tX
        \right\rVert_{L^4}
        \le \lVert X\rVert_{L^2},
        \qquad X\in Q(Y),
    \]
    because $\lambda_1<1$ and $(\lambda_m)_m$ is
    of finite length now.
    So $G = (Y, Q, \lambda)$ (with truncation by $K$)
    is actually $(2,p; t)$-hypercontractive
    and it extends to $G^{\otimes d}$ for any $d$
    by Theorem \ref{key-prop}
    (note that the truncation does not affect the
    random variables with $\tilde{\deg}_{G^{\otimes d}}X\le K$).
    Then, we finally use Proposition \ref{prop:gen-moment}
    to obtain the desired result with $C = K^t$.
\end{proof}

\subsection{Proof of Proposition \ref{prop:asp-suff}}
\begin{proof}
    Let $f\in L^2(\mu)$ be an eigenfunction
    with eigenvalue $\lambda\ge0$ of the integral operator,
    i.e., it satisfies
    $\int_\X k(x, y)f(y)\dd\mu(y) = \lambda f(x)$
    (assume this equality holds for all $x$, not just $\mu$-almost all).
    As Assumption \ref{assp:ker} is met from the general theory \citep{ste12},
    it suffices to show $\lambda \ge 1$ if and only if $f$ is constant.
    Note that $f=1$ is an eigenfunction for $\lambda=1$ by assumption.
    
    Assume $\lambda\ge1$.
    Since $k$ is bounded from the assumption, for an $(x_n)_{n=1}^\infty$
    converging to $x$,
    we have
    $f(x_n) = \frac1\lambda\int_\X k(x_n, y)f(y)\dd\mu(y)
    \to \frac1\lambda\int_\X k(x, y)f(y)\dd\mu(y) = f(x)$
    by the dominated convergence theorem.
    Thus, $f$ is continuous.
    Let $F = \max_{x\in\X}f(x)$.
    If $x^*\in f^{-1}(\{F\})$, then
    \[
        0 = F - f(x^*)
        = \int_\X k(x^*, y)\left(F - \frac1\lambda f(y)\right)\dd\mu(y).
    \]
    As $k(x^*, \cdot)$ is a probability density
    (recall $k\ge0$ from the assumption)
    with respect to $\mu$
    and $\supp\mu = \X$,
    we must have $\lambda \le 1$ and
    $k(x^*, y) = 0$ for all $y\not\in f^{-1}(\{F\})$.
    Now, it suffices to prove $f^{-1}(\{F\})=\X$ actually holds
    when $\lambda = 1$.
    Let $K = \max_{x, y\in\X}k(x, y)$.
    By taking an $\ve>0$ such that $\mu(f^{-1}([F-\ve, F))) \le 1/(2K)$,
    we have, for $x\not\in f^{-1}(\{F\})$,
    \begin{align*}
        f(x)
        &= \int_\X k(x, y)f(y)\dd\mu(y)\\
        &\le \int_{f^{-1}((-\infty, F-\ve))}k(x, y)f(y)\dd\mu(y)
        + \int_{f^{-1}([F-\ve, F))}k(x, y)f(y)\dd\mu(y)\\
        &\le (F - \ve)\int_{f^{-1}((-\infty, F-\ve))}k(x, y)\dd\mu(y)
        + F\int_{f^{-1}([F-\ve, F))}k(x, y)\dd\mu(y)\\
        &\le (F - \ve) + \ve\int_{f^{-1}([F-\ve, F))}k(x, y)\dd\mu(y)
        \le (F - \ve) + \frac\ve2 = F -\frac\ve2. 
    \end{align*}
    Therefore, if $f^{-1}(\{F\})=\X$, $f$ is disconnected
    (because $\X$ is path-connected),
    and it is contradiction.
    This completes the proof.
\end{proof}

\subsection{Proof of Proposition \ref{prop:ker-hc}}
We first prove the following lemma.
\begin{lem}
    For $p>2$, we have
    $    \lVert\K_0 f\rVert_{L^p}
        \le \lVert k_0\rVert_{L^p(\mu\otimes\mu)}
        \lVert f \rVert_{L^2}
    $
    for all $f\in L^2(\mu)$.
\end{lem}
\begin{proof}
    By Minkowski's integral inequality,
    we have
    \begin{align*}
        \lVert\K_0 f\rVert_{L^p}
        &=\left(\int_\X
            \left\lvert
                \int_\X k_0(x, y)f(y)\dd\mu(y)
            \right\rvert^p
        \dd\mu(x)\right)^{1/p}\\
        &\le \int_\X
            \left(
                \int_\X
                    \lvert k_0(x, y)f(y)\rvert^p
                \dd\mu(x)
            \right)^{1/p}
        \dd\mu(y)\\
        &\le \int_\X
            \left(
                \int_\X
                    \lvert k_0(x, y)\rvert^p
                \dd\mu(x)
            \right)^{1/p}
            \lvert f(y) \rvert
        \dd\mu(y)\\
        &\le
        \left(\int_\X
            \left(
                \int_\X
                    \lvert k_0(x, y)\rvert^p
                \dd\mu(x)
            \right)^{2/p}
        \dd\mu(y)\right)^{1/2}
        \lVert f\rVert_{L^2}
        \tag{Cauchy--Schwarz}\\
        &\le \lVert k_0\rVert_{L^p(\mu\otimes\mu)}
        \lVert f \rVert_{L^2}.
    \end{align*}
\end{proof}

From this lemma, we have
\begin{equation}
    \lVert e_m\rVert_{L^p}
    = \frac1{\sigma_m}\lVert \mathcal{K}_0e_m\rVert_{L^p}
    \le \frac{\lVert k_0\rVert_{L^p(\mu\otimes\mu)}}{\sigma_m}\lVert e_m\rVert_{L^2}
    \label{eq:lem-use}
\end{equation}
for each $m\ge2$.

\begin{proof}[Proof of Proposition \ref{prop:ker-hc}]
    It suffices to consider the case $\lVert k_0 \rVert_{L^4(\mu\otimes\mu)}<\infty$.
    Note that $\lambda_{\ell-1} = \sigma_\ell$ for $\ell=1,2,\ldots$
    for the GRP $G_{k, \mu}$,
    so $\lambda_1 = \sigma_2 = \lVert \mathcal{K}_0\rVert$.
    
    Let $r_0$ be the minimum nonnegative number satisfying
    $\lVert \mathcal{K}_0 \rVert^{-r_0} \ge \sqrt{3}\tr(\mathcal{K}_0)$.
    Then, for $r:=1+r_0$, we have
    \begin{equation}
        \sum_{\ell=2}^\infty\sigma_\ell^r
        \le \sigma_2^{r_0}\sum_{\ell=2}^\infty\sigma_\ell
        =\lVert\mathcal{K}\rVert^{r_0}\tr(\mathcal{K}_0)
        \le\frac1{\sqrt{3}}
        \label{eq:pf-ker-1}
    \end{equation}
    Let $s_0$ be the minimum nonnegative number satisfying
    $\lVert \mathcal{K}_0 \rVert^{-s_0} \ge \lVert k_0 \rVert_{L^4(\mu\otimes\mu)}$.
    As $\lVert\mathcal{K}_0\rVert \in (0, 1)$ from Assumption \ref{asp:grp},
    $s_0$ is well-defined.
    Then, for $s:=1+s_0$ and $m\ge2$, from \eqref{eq:lem-use}, we have
    \begin{equation}
        \lVert e_m\rVert_{L^4} \le
        \frac{\lVert k_0\rVert_{L^4(\mu\otimes\mu)}}{\sigma_m}\lVert e_m\rVert_{L^2}
        \le \frac1{\sigma_m\lVert \mathcal{K}_0\rVert^{s_0}}\lVert e_m\rVert_{L^2}
        \le \sigma_m^{-1-s_0}\lVert e_m\rVert_{L^2}.
        \label{eq:pf-ker-2}
    \end{equation}
    
    Thus, the condition for $s$ and $t:=r+s$ of Proposition \ref{prop:prod-hc}
    is satisfied, and so we have the desired conclusion.
\end{proof}

\end{document}